\documentclass[10pt]{amsart}
\usepackage[normalem]{ulem}
\usepackage{bbm}
\usepackage[usenames,dvipsnames,svgnames,table]{xcolor}
\usepackage[pagebackref,colorlinks=true,linkcolor=BrickRed,citecolor=OliveGreen,pdfstartview=FitH]{hyperref}
\usepackage{graphicx}
\usepackage{float}
\usepackage{comment}
\usepackage{enumitem}
\usepackage{marginnote}
\usepackage{mathrsfs}
\usepackage{xcolor}

\newtheorem{thm}{Theorem}[section]

\newtheorem{coro}[thm]{Corollary}
\newtheorem{lem}[thm]{Lemma}
\newtheorem{prop}[thm]{Proposition}

\theoremstyle{definition}

\theoremstyle{remark}

\newtheorem{remk}[thm]{Remark}

\newcommand{\im}{{\rm im}}
\newcommand{\vol}{{\rm vol}}

\newcommand{\tgamma}{\widetilde{\gamma}}

\newcommand{\Rbb}{ {\mathbb R}}
\newcommand{\Zbb}{ {\mathbb Z}}

\newcommand{\Hcal}{ {\mathcal H}}
\newcommand{\Mcal}{ {\mathcal M}}
\newcommand{\Dcal}{ {\mathcal D}}

\newcommand{\Ucal}{ {\mathcal U}}

\newcommand{\Vcal}{ {\mathcal V}}

\newcommand{\Ycal}{ {\mathcal Y}}

\newcommand{\Scal}{ {\mathcal S}}

\newcommand{\Pcal}{ {\mathcal P}}

\newcommand{\Tcal}{ {\mathcal T}}
\newcommand{\Rcal}{ {\mathcal R}}

\newcommand{\Wcal}{ {\mathcal W}}

\newcommand{\Sbb}{{\mathbb S}}

\newcommand{\spec}{{\rm spec}}

\renewcommand{\phi}{\varphi}

\usepackage{mathabx}
\usepackage{amssymb}
\usepackage{mathrsfs}

\definecolor{olive}{rgb}{0.13, 0.55, 0.13}

\definecolor{olive}{rgb}{0.5, 0.5, 0.0}

\definecolor{magenta}{rgb}{1.0, 0.0, 1.0}
\newcommand{\craig}[1]{\textcolor{magenta}{#1}}

\title[Multiplicity and nodal domains of torus invariant metrics]{Spectral 
multiplicity and nodal sets for generic torus-invariant metrics}

\author{Donato Cianci} 
\address{GEICO, Chevy Chase, MD}
\email{\href{mailto:dcianci@geico.com}{dcianci@geico.com}}

\author{Chris Judge$^\dagger$}
\address{Indiana University, Department of Mathematics, Bloomington, IN 47405} 
\email{\href{mailto:cjudge@indiana.edu}{cjudge@indiana.edu}}
\thanks{$^\dagger$ The work of C.J.\ is partially supported by a Simons 
Foundation Collaboration Grant. }

\author[S. Lin]{Samuel Lin}
\address{University of Oklahoma\\ Department of Mathematics \\ Norman, OK 73071}
\email{\href{mailto:samuel.z.lin-1@ou.edu}{samuel.z.lin-1@ou.edu}}

\author[C. Sutton]{Craig Sutton$^\sharp$}
\address{Dartmouth College\\ Department of Mathematics \\ Hanover, NH 03755}
\email{\href{mailto:craig.j.sutton@dartmouth.edu}{craig.j.sutton@dartmouth.edu}}
\thanks{$^\sharp$ The work of C.S. is partially supported by a Simons Foundation Collaboration Grant}

\begin{document}

\maketitle

\begin{abstract}
Let a torus $T$ act freely on a closed manifold $M$ of dimension at least two. 
We demonstrate that, for a generic $T$-invariant Riemannian metric $g$ on $M$,
each real $\Delta_g$-eigenspace is an irreducible real representation of $T$ and,
therefore, has dimension at most two. We also show that,
for the generic $T$-invariant metric on $M$, if
$u$ is a non-invariant real-valued $\Delta_g$-eigenfunction that vanishes 
on some $T$-orbit, then the nodal set of $u$ is a connected smooth hypersurface 
whose complement has exactly two connected components. 
\end{abstract}

\section{Introduction}

Karen Uhlenbeck \cite{Uhlenbeck} proved that, for the
generic metric $g$ on a smooth closed  manifold $M$ of dimension $n >1$,
the associated Laplacian $\Delta_g$ acting on functions has 
one-dimensional eigenspaces. On the other hand, metrics that are invariant 
under a group action often have Laplace eigenspaces with
dimension greater than one.
For example, if a torus acts freely on $M$ and preserves $g$, 
then each non-invariant, real-valued eigenfunction of $\Delta_g$
lies in an eigenspace of dimension at least two. 
Here we 
show that for the generic torus invariant metric, the dimension of 
each eigenspace is at most two.

\begin{thm}
\label{thm:main}
Let $M$ be a smooth connected closed manifold of dimension $n>d$
such that the $d$-dimensional torus $T^d$ acts freely 
and smoothly on $M$.
Then for each $\ell \geq 2$,
there exists a residual subset of the space of the torus\craig{-}invariant 
$C^{\ell}$ Riemannian metrics on $M$ such that 
if $g$ lies in this residual set, then: 

\begin{enumerate}

\vspace{.1cm}

\item  Each (real) eigenspace of $\Delta_g$ has dimension at most two.
More precisely, if an eigenfunction $\phi: M \to \Rbb$
is torus-invariant, then $u$ lies in a one-dimensional eigenspace 
and otherwise $\phi$ lies in a 2-dimensional eigenspace on which
$T^d$ acts irreducibly.

\vspace{.25cm}

\item The nodal set $\phi^{-1}(0)$ of each eigenfunction $\phi$ is a smooth
      hypersurface.
\end{enumerate}
\end{thm}

We also produce new examples of Riemannian manifolds 
with infinitely many eigenfunctions having exactly two nodal domains.
For example, we prove

\begin{thm}
\label{thm:nodal=2}
Suppose $B$ is a connected smooth closed manifold 
so that $H_2(B, \Zbb)$ contains no non-trivial element of finite order.
Let $M$ be the total space of a nontrivial oriented circle
bundle $M \to B$.
Then for the generic circle-invariant metric $g$  
on $M$, each non-invariant eigenfunction of $\Delta_g$ 
has exactly two nodal domains.
\end{thm}

Junehyuk Jung and Steve Zelditch \cite{Jung-Zelditch} proved both
Theorem \ref{thm:main} and Theorem \ref{thm:nodal=2}
in the  case where $M$ is the unit cotangent bundle of an
oriented surface $X$ 
and the metrics on $M$ are chosen so that the fibers are geodesic
with constant length.\footnote{
Such metrics are sometimes called Kaluza-Klein metrics by physicists.
See e.g. \cite{Wesson}.}

Theorem \ref{thm:nodal=2} is remarkable because it provides a large 
class of Riemannian manifolds with a sequence $\{\phi_k\}$ of eigenfunctions 
that span a space of Weyl density 1 so that each $\phi_k$ has exactly two nodal
domains (see Proposition \ref{prop-Weyl}).
Prior to \cite{Jung-Zelditch}, both the square and the 
standard two-dimensional sphere were shown to possess a sequence of eigenfunctions
with exactly two nodal domains  \cite{Stern-thesis} \cite{Courant-Hilbert} 
\cite{Lewy}.\footnote{See  \cite{Berard-Helffer-square}
and \cite{Berard-Helffer-spherical}  for interesting discussions of 
these constructions.} However, these sequences span a space of Weyl density zero 
and their construction relies on the existence 
of high-dimensional eigenspaces.

Theorem \ref{thm:nodal=2} applies to some well-known manifolds. 
For example, each odd-dimensional sphere $S^{2n+1}$ is the total 
space of the oriented circle bundle associated to the tautological 
complex line bundle over the $n$-dimensional complex projective space. 
When $n=1$, this is the well-known Hopf fibration of $S^3$. 
Theorem \ref{thm:main} and Corollary \ref{thm:nodal=2} 
imply that for a generic circle-invariant metric on $S^{2n+1}$, 
each eigenvalue has multiplicity at most two and each 
non-invariant real eigenfunction has exactly two nodal domains.
In comparison, the standard metric on $S^{2n+1}$ is invariant,
but the dimension of the eigenspace 
and the number of nodal domains of the typical eigenfunction 
grows with the eigenvalue \cite{Nazarov-Sodin} \cite{Nazarov-Sodin-2}. 
Jung and Zelditch \cite{Jung-Zelditch-sphere}
recently showed, however, that, for the standard metric on $S^3$, 
the typical eigenfunction within a fixed nontrivial isotypical component 
has exactly two nodal domains.

The proof of Theorem \ref{thm:main} consists of analyzing the 
restriction, $\Delta_{g, \alpha}$, of the Laplacian to the
space, $H^k_{\alpha}$,  consisting of vector-valued
Sobolev $H^k$ functions 
$u:M \to \Rbb^2$
that satisfy  $u(\theta^{-1} \cdot x) = R_{\alpha \cdot \theta} \cdot u(z)$ for each $\theta \in T$ where $R_{\psi}$ is rotation 
by angle $\psi$ 
(see \S \ref{sec:prelim}). The proof begins in \S \ref{sec:variational} 
where we provide some variational formulas for the Laplacian 
and its eigenvalues.  The variational formula for eigenvalues is used 
in \S \ref{sec:separation} to show that, for the generic torus-invariant metric, 
the spectra of $\Delta_{g, \alpha}$ and $\Delta_{g, \beta}$ are disjoint 
if $\alpha \neq \pm \beta$ (Theorem \ref{thm:perturb}). 
In \S \ref{sec:simplicity} we use the method 
of Uhlenbeck \cite{Uhlenbeck} to show that the spectrum of $\Delta_{g, \alpha}$ 
is simple (Theorem \ref{thm:simple-on-weights}). 
At the end of \S \ref{sec:simplicity}, we combine 
Theorem \ref{thm:perturb} 
and Theorem \ref{thm:simple-on-weights} to prove part (1) of Theorem \ref{thm:main}.
In \S \ref{sec:nodal-sets} we use the method of Uhlenbeck to show that, for a generic 
torus-invariant metric,  
zero is a regular value of each eigenfunction of $\Delta_{g, \alpha}$.
We explain how this implies part (2) of Theorem \ref{thm:main}
as well as Theorem \ref{thm:nodal=2}.

Our overall approach mirrors the approach in \cite{Jung-Zelditch}. However, there
are some important differences. Jung and Zelditch use the natural unitary 
isomorphism between $H^0_{\alpha}$ and the $L^2$-sections of 
$\kappa^{\otimes \alpha}$ for $\alpha \neq 0$. Under this isomorphism, the operator
$\Delta_{g, \alpha}$ corresponds to a Bochner Laplacian acting on sections of 
$\kappa^{\otimes \alpha}$. Jung and Zelditch show that the spectra of the associated 
Bochner Laplacians are generically disjoint for $\alpha \neq \beta$, 
and they use Uhlenbeck's method to show that the spectrum of each is simple. 
Moreover, they prove these genericity results for the Bochner Laplacian in the 
larger context of holomorphic line bundles with generic base metrics, 
generic hermitian metrics, and/or generic connections. However, 
the equivalence of $\Delta_{g,\alpha}$ and the Bochner Laplacian is 
proven only in the context of canonical bundles (see Lemma 6.6 \cite{Jung-Zelditch}).

In this paper we analyse the operators $\Delta_{g, \alpha}$ directly, 
and we do not consider Bochner Laplacians. We also provide a simpler proof of 
the fact that $\{ x \in M\, :\, u(x) \neq 0\}$ has two components for a non-invariant real-valued eigenfunction $u$.
Whereas the proof in \cite{Jung-Zelditch} uses an intricate combinatorial analysis
of the nodal domains, we show directly that each nodal set is connected for the 
generic invariant metric. Thus, because zero is a regular value for the generic invariant
metric, the nodal set is a connected smooth hypersurface and hence its complement 
has at most two components.

It is natural to ask whether an analogue of Theorem \ref{thm:main}   
holds for other groups $G$. In general, if $G$ preserves a metric $g$ on $M$, then 
each real eigenspace $E$ of $\Delta_g$ is a $G$-invariant subspace of 
the space of real-valued $L^2$ functions. Theorem \ref{thm:main} is equivalent 
to the statement that if $G$ is a torus, then each real eigenspace is 
irreducible for the generic $G$-invariant metric. Similar results have
been obtained in the case of finite group actions \cite{Zelditch}, left-invariant
metrics on $SU(2)$ \cite{Schueth}, and invariant 
metrics on trivial $SU(2)$ bundles \cite{MrcGmz19}.
These results support the belief in quantum mechanics that
the eigenspaces of the typical Hamiltonian should be irreducible. 
Physicists regard an eigenvalue associated to a
reducible eigenspace as being
`accidentally degenerate' \cite{Wigner} \cite{Faddeev}.  

\begin{remk}
Building upon the techniques and results of \cite{Jung-Zelditch}, 
Junehyuk Jung  and Steve Zelditch recently posted a preprint
\cite{Jung-Zelditch-22} with 
results similar to those contained in this article.
While Jung-Zelditch and we both use Uhlenbeck's framework, our implementations are very different. 
\end{remk}


\section{Preliminaries}
\label{sec:prelim}

We will let $H^k$ denote the Sobolev space of real-valued
functions on $M$
whose weak derivatives of order up to $k$ are square-integrable in 
each coordinate chart with respect to Lebesgue measure. 
In particular $L^2= H^0$.

Let $T^d := (\Rbb/2 \pi \Zbb)^d$ 
denote the $d$-dimensional torus.  
Given $\theta \in T^d$, let $\phi_\theta$ denote the 
diffeomorphism $\phi_\theta(x) = \theta \cdot x$.
The torus $T^d$ acts on the real vector space 
$H^k$ by 
\begin{equation}
\label{eqn:action}
(\theta \cdot u)(x)~
=~
u(\theta^{-1} \cdot x)~ 
=
~ u \circ \phi_{\theta}^{-1} (x).
\end{equation}
In the language of representation theory, 
this action is called the (left) regular representation.
Let $d \mu_\theta := d \theta_1  \cdots d \theta_d$.

Define $H_{0}^k$ to be the space of $H^k$
functions that are invariant under the action of $T^d$, that is,
$\theta \cdot u = u$ for each $\theta \in T^d$.  Integration over 
the orbits defines a projection $\pi_0: H^k \to H^k_0$ defined by
\[
\pi_0(u)(x)~
=~
\int_T  (\theta \cdot u)(x)\, d\mu_{\theta}.
\]

For each nonzero $\alpha \in \Zbb^d$, 
define the space $H_{\alpha}^k$ to be the 
space of $\Rbb^2$-valued functions $u=(u_1,u_2)^{t}:M \to \Rbb^2$ such that 
$u_1,u_2 \in H^k$ and for each $\theta \in T^d$ and $ x \in M$
\begin{equation}
\label{eqn:twist}
\left(
\begin{array}{c}
u_1(\theta^{-1} \cdot x)\\
u_2(\theta^{-1} \cdot x)
\end{array}
\right)~ 
=~ 
R_{\alpha \cdot \theta} 
\cdot
\left(
\begin{array}{c}
 u_1(x) \\
 u_2(x) 
\end{array}
\right)
\end{equation}
where
\begin{equation}
\label{eqn:R}
R_{\psi}~
=~
\left(
\begin{array}{cc}
\cos(\psi) & \sin(\psi)  \\
-\sin(\psi) & \cos(\psi)
\end{array}
\right).
\end{equation}
The space $H^{k}_{\alpha}$ is invariant under the 
$T^d$-action $\theta \cdot (u_1,u_2)= (\theta \cdot u_1, \theta \cdot u_2)$.

From (\ref{eqn:twist}) we see that if $\theta \cdot \alpha =\pi/2 \mod 2 \pi$,
then $\theta \cdot u_1 = u_2$. In particular, the map 
$\iota_{\alpha}: H^k_{\alpha} \to H^k$ 
defined by $\iota_{\alpha}(u_1,u_2)= u_1$ is injective and the 
image is invariant under $T^d$.
For $\alpha=0$, let $\iota_{\alpha}:H_0^k \to H^k$ denote the inclusion map.
If $\alpha = - \beta$, then 
$\iota_{\alpha}(H^k_{\alpha})= \iota_{\alpha}(H^k_{\beta})$, and 
if $\alpha \neq \pm \beta$, then 
$\iota_{\alpha}(H^k_{\alpha}) \cap \iota_{\alpha}(H^k_{\beta})= \{0\}$.

For $\alpha \neq 0$, the space $H^k_{\alpha}$ is the image of the map
$\pi_{\alpha}: H^k \to H^k_{\alpha}$ defined by
\[
\pi_{\alpha}(u)~
=~
\left(
\begin{array}{c}
\int_{T} (\theta \cdot u)(x) \cdot \cos(\alpha \cdot \theta)\, d \mu_\theta 
\vspace{.25cm}\\ 
\int_{T} (\theta \cdot u)(x) \cdot (-\sin(\alpha \cdot \theta))\, d \mu_\theta
\end{array}
\right).
\]
Using this `Fourier coefficient map' one finds that $H^k$ is the direct sum
\[ 
H^k~ 
=~ 
H_0^k\, \oplus\,  \left(\bigoplus_{\alpha} \iota_{\alpha}(H^k_{\alpha}) \right)
\]
where $\alpha$ ranges over a set of orbit representatives 
of the action of $\Zbb/2 \Zbb$ on $\Zbb^d$ given by 
$m \cdot \alpha = (-1)^m \cdot \alpha$.
Indeed, 
if $v \in H^k$ were orthogonal to $\iota_{\alpha}(H_{\alpha}^k)$ for 
each $\alpha$, then the restriction of $v$ to almost every 
$T^d$ orbit would vanish by the standard theory of Fourier series.

If $e_j$ is the standard basis vector on $\Rbb^d$,
then the path $t \mapsto t \cdot e_j$ defines
a one-parameter subgroup of $T^d$.
Define the smooth vector field $\partial_j$ on $M$ by
\begin{equation}
\label{eqn:partial_j}
(\partial_j u) (x)~
=~
\left. \frac{\partial}{\partial t} \right|_{t=0}
 u\left(  (t \cdot e_j) \cdot x   \right),
\end{equation}
where $u \in C^{\infty}$.
Note that $(\phi_\theta)_*(\partial_j|_{x})= \partial_j |_{\theta\cdot x}$, 
which implies that
$\partial_j \circ \partial_k = \partial_k \circ \partial_j$.
Moreover,  if $u \in H^k_{\alpha }$, then it follows from 
(\ref{eqn:partial_j}) that
\begin{equation}
\label{eqn:vert-derivative}
\partial_{j}u~ 
=~
\alpha_j 
\cdot
\left(\begin{array}{cc}
0 & 1 \\
-1 & 0 
\end{array}
\right)
\cdot  
u
\end{equation}
and hence $-\partial^2_j u = \alpha_j^2 \cdot u$.

Associated to each metric $g$ on $M$, there is a Laplace-Beltrami 
operator $\Delta_g: H^k \to H^{k-2}$ defined implicitly by 
\begin{equation}
\label{eqn:Laplacian}
\int (\Delta_g u) \cdot v~  d\nu_{g}~
=~
\int g \left( \nabla_g u, \nabla_g v \right)\, d\nu_g
\end{equation}
where $\nabla_g$ is the Riemannian gradient and $d\nu_g$ 
is the Riemannian measure on $M$. Because $M$ is compact
and $\Delta_g$ is elliptic,  
Rellich's lemma implies that the resolvent $(\Delta - \lambda)^{-1}$ 
is compact and hence the  spectrum of $\Delta_g$ consists
of isolated eigenvalues with finite dimensional eigenspaces.

Now suppose that $g$ is $T$-invariant.
Then $\Delta_g \circ \phi_\theta^* =  \phi_\theta^* \circ \Delta_g$,
and it follows that $\Delta_g$ preserves $\iota_{\alpha}(H^k_{\alpha})$.
In particular, for $u=(u_1,u_2)^t \in H^{k}_{\alpha} $, we may define
 $\Delta_{g,\alpha}: H^{k}_{\alpha} \to H^{k}_{\alpha}$ by 
\[
 \Delta_{g,\alpha} 
\left(
\begin{array}{c}
u_1  \\
u_2
\end{array}
\right)~
=~
\left(
\begin{array}{c}
\Delta_{g} u_1  \\
\Delta_{g} u_2
\end{array}
\right)
\]
If $W$ is an eigenspace of $\Delta_{g, \alpha}$ 
then $\iota_{\alpha}(W)$ is of an eigenspace of $\Delta_g$ with the same eigenvalue.
Thus, the spectrum of $\Delta_{g,\alpha}$ is a subset
of the spectrum of $\Delta_g$ and hence is discrete.

\vspace{.5cm}

We conclude this section by explaining the claim 
in the introduction that the Weyl density of the non-invariant 
eigenfunctions equals one. First, we make precise the notion of 
Weyl density.  For each $x>0$, let $E_x$ denote the direct sum
of the eigenspaces of $\Delta_g$ whose associated eigenvalues are at most $x$. 
The Weyl density of a subspace $V \subset L^2$ equals 
\[  \limsup_{x \to \infty}   \frac{\dim(V \cap E_x)}{\dim(E_x)}.
\]
The following is a consequence of, for example, \cite{Donnelly},
but we provide a simple explanation for the convenience of the reader.

\begin{prop}
\label{prop-Weyl}
The invariant subspace $L^2_0$ has Weyl density equal to zero,
and hence the space of non-invariant functions has Weyl density equal to 1.
\end{prop}

\begin{proof}
Since the torus action is smooth and free, the quotient $M/T^d$ 
is a smooth $n-d$-dimensional 
manifold and the quotient map $\pi: M \to M/T$ is a smooth
submersion. Since the action preserves $g$, the metric $g$ descends 
to a metric $g'$ on $M\setminus T$ and $\pi_*(d\nu_g)=d \nu_{g'}$.
Let $d\nu_{g'}$ be the associated Riemannian measure on $M/T$
and let $H^0(M/T)$ denote the associated Hilbert 
space of square-integrable functions.
For each $x \in M$, let $\vol(T \cdot x)$ denote the $d$-dimensional measure of the orbit $T \cdot x$
with respect to $g$.  
The map $\Phi:H^0(M/T) \to H^0_0$ defined by 
\[  \Phi(u)(x)~ =~ \vol(T \cdot x)^{-\frac{1}{2}} \cdot  (u \circ \pi)(x).
\]
is a unitary isomorphism. The operator $P:=\Phi^{-1} \circ \Delta_g \circ \Phi$
is a nonnegative second order elliptic differential operator on $M/T$, 
and hence Weyl's law implies that $\dim(E'_x)$ is $O(x^{\frac{n-d}{2}})$
where $E'_x$ is the direct sum of the eigenspaces of $P$ whose eigenvalues
of size at most $x$. We have $\Phi^{-1}(E'_x)= E_x \cap L^2_0$. 
On the other hand, Weyl's law implies that $\dim(E_x) \sim c \cdot x^{\frac{n}{2}}$ where $c>0$. 
\end{proof}



\section{Some variational formulas}
\label{sec:variational}

In this section we derive a general formula for 
the first variation of the Laplacian under metric perturbations.
We then specialize this formula to perturbations associated 
to a decomposition of the tangent bundle into orthogonal subbundles. 
At the end of the section, we 
specialize further to the special case of 
perturbations of $T$-invariant metrics.


\subsection{General variational formulas}

We will let $\Mcal_{\ell}$  denote the set of $C^{\ell}$ 
Riemannian metric tensors on $M$. The space 
$\Mcal_{\ell}$ is an open convex subset of the Banach space $\Scal_{\ell}$  
of  $C^{\ell}$ symmetric $(0,2)$ tensors on $M$.\footnote{A norm on 
$\Scal_{\ell}$ can be constructed using, 
for example, a choice of 
a smooth (co)metric on $M$.} 
We fix $\ell \geq 2$, and so in the sequel
we will suppress the subscript $\ell$ from notation.

Given a differentiable path $t \mapsto g_t \in \Mcal$, let $d\nu_t$, $\nabla_t$,
and $\Delta_t$ denote respectively the measure, gradient, and Laplacian 
associated to $g_t$. 
In what follows, a dot above a symbol will indicate the first variation 
$\left. \partial_t \right|_{t=0}$. For example, $\dot{\lambda}$
will indicate $\left. \partial_t \right|_{t=0} \lambda_t$
and $d\dot{\nu}$ indicates $\left. \partial_t \right|_{t=0} d\nu_t$.

\begin{lem}
\label{lem:Lap-perturbation}
Let $t \mapsto g_t \in \Mcal$ be a differentiable path of metrics.
Then for each $u,v \in H^2$ we have 
\begin{equation}
\label{eqn:pert-Laplacian}
\int_M (\dot{\Delta} u)v\, d\nu_0\,
=\,
-\int_M \dot{g} \left(\nabla_0 u, \nabla_0\, v \right)\, d\nu_0~
+~
\int_M \left(g_0(\nabla_0 u,\nabla_0 v) -  (\Delta_0 u) v \right)\, 
d\dot{ \nu}.
\end{equation}
\end{lem}

\begin{proof}
By (\ref{eqn:Laplacian}) we have 
\begin{equation}
\label{eqn:Lap-var}
\int_M (\Delta_t u) \cdot v \, d \nu_t~
=~
\int_M  g_t \left(\nabla_t u, \nabla_t v \right)\, d \nu_t. 
\end{equation}
Since $g_t(\nabla_t w,X)= Xw$ for each $X$ and $w$, we find that
$g(\dot{\nabla} w, X)= - \dot{g}(\nabla w, X)$.
Thus by differentiating both sides of (\ref{eqn:Lap-var})
with respect to $t$ and setting $t=0$ we obtain
\[ 
\int (\dot{\Delta} u) v \, d \nu_0~
+~
\int (\Delta_0 u) v \, d \dot{\nu}~
=~
-\int \dot{g} (\nabla_0 u, \nabla_0 v)\, d\nu_0~
+~
\int g(\nabla_0 u, \nabla_0 v)\, d \dot{\nu}.
\]
The claimed formula follows.
\end{proof}

\begin{coro}
Let $t \mapsto g_t \in \Mcal$ be a differentiable path of metrics.
Suppose that $t \mapsto u_t$ is
a differentiable path of $H^k$ functions such that  
$\Delta_t u_t = \lambda_t u_t$. 
Then 
\begin{equation}
\label{eqn:pert}
\dot{\lambda}
\int u_0^2\, d\nu_0
=
-\int \dot{g} \left(\nabla_0\, u_0, \nabla_0\, u_0 \right) d\nu_0
+
\int \left(g_0(\nabla_0 u_0,\nabla_0 u_0) - \lambda_0 \cdot u_0^2 \right) 
d\dot{ \nu}.
\end{equation}

\end{coro}

\begin{proof}
We have $\Delta_t u_t = \lambda_t  u_t$, and hence differentiation
gives 
$\dot{\Delta} u_0 + \Delta_0 \dot{u}=  \dot{\lambda}u + \lambda_0 \dot{u}$.
Integrate both sides of this equation again $u_0$ and then apply 
Lemma \ref{lem:Lap-perturbation}.
\end{proof}


\subsection{Perturbations associated to orthogonal subbundles}

We next construct  metric perturbations of a given metric $g$
associated to a decomposition of $TM$ into two $g$-orthogonal subbundles. 
That is, we suppose that we are given two  
smooth subbundles $\Vcal$ and $\Hcal$ of the tangent bundle so that 
$T_p M = \Vcal_p \oplus \Hcal_p$ is a $g$-orthogonal direct sum
for each $p \in M$.

Our first class of perturbations consists of rescaling $g$ along 
$\Hcal$ and $\Vcal$ indpependently. To be precise, for each pair,
$a$ and $b$, of $C^{\ell}$  functions on $M$,  define
a new metric $g_{a,b}$ on $TM = \Vcal \oplus \Hcal$ by setting  
\begin{equation}
\label{eqn:split-conformal}
g_{a,b}(v+h,v'+h')~ 
=~  
a \cdot g(v,v')~ +~ b \cdot g(h, h')
\end{equation}
for each $v,v' \in \Vcal$ and $h,h' \in \Hcal$.
Note that $\Vcal \oplus \Hcal$ remains a $g_{a,b}$-orthogonal direct sum.

Let $j$ (resp. $k$) denote the (constant) dimension of 
$\Vcal_p$ (resp. $\Hcal_p$).
In particular, $j+k=n$ where $n$ is the dimension of $M$.
Given $ u \in C^{\infty}$, let $\nabla^\Vcal_g u$ (resp. $\nabla^\Hcal_g u$)
denote the orthogonal projection of $\nabla_g u$ onto $\Vcal$ (resp. $\Hcal$).

\begin{lem}
\label{lem:split-rescaling}
Let $t \mapsto a_t$ and $t \mapsto b_t$ 
be differentiable paths in $C^{\ell}$ 
such that $b_0=a_0=1$.
Let $g_t=g_{a_t, b_t}$ denote the associated path of Riemannian metrics 
defined by \eqref{eqn:split-conformal}.  
For each $u, v \in C^{\infty}$, we have 
\begin{eqnarray*} 
2 \int \left( \dot{\Delta} u \right) v\, d\nu_{0}
&=& 
-\int \left( \Delta_{0} u \right) v 
\cdot 
\left( j \cdot  \dot{a} + k \cdot \dot{b} \right)\, d\nu_{0} \\
&&
+
\int g_0\left( \nabla^\Vcal_0 u, \nabla^\Vcal_0 v \right) 
\cdot 
\left( (j-2) \cdot \dot{a} +k \cdot \dot{b} \right)\, d\nu_{0} \\
&&
+
\int g_0\left( \nabla^\Hcal_0 u, \nabla^\Hcal_0 v \right) 
\cdot 
\left( j \cdot \dot{a} + (k-2) \cdot  \dot{b} \right)\, d\nu_{0}.
\end{eqnarray*}
\end{lem}

\begin{proof}
We will apply Lemma \ref{lem:Lap-perturbation}.
First note that $d\nu_{t}=a_t^{j/2} \cdot b_t^{k/2} \cdot d\nu_{0}$,
and hence since $a_0=1=b_0$ we have $d \dot{\nu} = (j/2)\dot{a} + (k/2) \dot{b}$. 
Because $\Vcal$ and $\Hcal$ are $g_t$-orthogonal for each $t$, 
we find that 
\[
g_0(\nabla_0 u, \nabla_0 v)~
=~
g_0(\nabla^{\Hcal}_0 u, \nabla^{\Hcal}_0 v)~
+~
g_0(\nabla^{\Vcal}_0 u, \nabla^{\Vcal}_0 v)
\]
and
\[ 
\dot{g}(\nabla_0 u, \nabla_0 v)~
=~
\dot{a} \cdot g_0(\nabla^{\Vcal}_0 u, \nabla^{\Vcal}_0 v)+
\dot{b} \cdot g_0(\nabla^{\Hcal}_0 u, \nabla^{\Hcal}_0 v).
\] 
The claim now follows from Lemma \ref{lem:Lap-perturbation}.
\end{proof}

Next we consider a somewhat different type of metric variation.
Let $X$ be a section of $\Hcal$ and let $Y$ be a section of $\Vcal$.  
Let $\xi$ (resp. $\eta$)
be the  dual one form defined by 
$\xi(Z) =g(X, Z)$  (resp. $\eta(Z) =g(Y,Z)$) for each $Z$.
Define
\[
g_{X,Y}~
=~
g~
+~
\xi \otimes \eta~
+~
\eta \otimes \xi.
\]

\begin{lem}
\label{lem:mixed-perturbation}
Consider the family $t \mapsto g_{tX,tY}$. 
For each smooth $u$, $v$, we have 
\begin{equation}
\label{eqn:mixed-perturbation}
\int_M \dot{\Delta} u \cdot v\, d\nu_0~
=~
-\int_M
\left( 
(Xu) \cdot (Y v)~ +~ (Y u)  \cdot (X v )
\right)\, 
d\nu_0.
\end{equation}
\end{lem}

\begin{proof}
We apply Lemma \ref{lem:Lap-perturbation}. 
Because $X$ and $Y$ are $g$ orthogonal,
we find that $d\dot{\nu} = 0$.
We have $\dot{g} = \xi \otimes \eta +   \eta \otimes \xi$.
We have $\xi(\nabla u)= g(\nabla u, X) = Xu$ and 
$\eta(\nabla v) = g(\nabla u, Y) = Yv$  and so 
$\xi \otimes \eta(\nabla u, \nabla v) = Xu \cdot Y v$.
Similarly, $\eta \otimes \xi(\nabla u, \nabla v) = Yu \cdot X v$.
The claim follows. 
\end{proof}


\subsection{Metrics invariant under a torus action} \label{Subsection:MetricTorusInvariant}

In the following we will specialize to metrics that are 
invariant under a fixed action of the $d$-dimensional torus $T$.
Let $\Vcal$ be the subbundle consisting of vectors tangent to the $T$ action.  
let $\Hcal$ be the subbundle consisting of vectors 
that are $g$-orthogonal to $\Vcal$ where $g$ is a fixed  $T$-invariant metric.

Let us specialize Lemma \ref{lem:split-rescaling}.
We will assume that the smooth functions $a$ and $b$ in (\ref{eqn:split-conformal})
are $T$-invariant.  This implies that the metric $g_{a,b}$ is $T$-invariant.
The gradient $\nabla^\Vcal$ is a linear combination of the vector fields
$\partial_j$ given by (\ref{eqn:partial_j}), and hence 
$\nabla^{\Vcal}$ is well understood. On the other hand,
$\nabla^\Hcal$ is less well-understood in general. For this reason
we choose $(j-2) \cdot \dot{a} +k \cdot \dot{b}=0$ so as to 
eliminate the term in Lemma \ref{lem:split-rescaling} that involves
$\nabla^\Hcal$. Noting that $j=d$ and $k=n-d$ we obtain the following.

\begin{lem}
\label{lem:variation}
If $g$, $a$, and $b$ are $T$-invariant and 
$d \cdot \dot{a} + (n-d-2) \cdot \dot{b}=0$, then 
\[
\int \left( \dot{\Delta} u \right) v \, d\nu_0~
=~
- \int \dot{b} \cdot \left( \Delta_0 u \right) v\, d\nu_0~
+~
\frac{n-2}{d} \int \dot{b} \cdot 
g\left( \nabla^\Vcal_0 u, \nabla^\Vcal_0 v \right)\, d\nu_0.
\]
\end{lem}


\section{The spectra of $\Delta_{g, \alpha}$ and $\Delta_{g, \beta}$ are generically disjoint}
\label{sec:separation}

Recall that $\Delta_{g,\alpha}$ is the restriction of the Laplacian 
to the space $H^k_{\alpha}$ where $\alpha \in \Zbb^d$.
Let ${\rm \spec}(\cdot)$ denote the spectrum of an operator.
The purpose of this section is to prove the following: 

\begin{thm}
\label{thm:perturb} 
 If $\alpha \neq \pm \beta \in \mathbb{Z}^d$, 
then the set 
of metrics $g$ for which 
${\rm \spec}(\Delta_{g,\alpha}) \cap {\rm \spec}(\Delta_{g,\beta}) = \emptyset$ 
 is residual in $\Mcal^T$.
\end{thm}

Our proof of Theorem \ref{thm:perturb}  
is based on analytic perturbation theory \cite{Kato}.
In this section $t \mapsto g_t \in \Mcal^T$ 
will denote an analytic path of torus-invariant 
metrics defined for $t \in (-\delta, \delta)$
where $\delta>0$.

\begin{lem}
\label{lem:weight-analytic}
For each $t$ there exists a $g_t$-orthonormal basis $\{\phi^{\alpha}_{j,t}\}$
of $H^k_{\alpha}$ consisting of eigenfunctions of $\Delta_{g_t,\alpha}$ 
so that $t \mapsto \phi^{\alpha}_{j,t}$ is analytic for each $j$.  
\end{lem}

\begin{proof}
The operator $\Delta_{g_t, \alpha}: H^0_{\alpha} \to H^0_{\alpha}$ 
is self-adjoint with domain $H^2_{\alpha}$. 
For each fixed $u \in H^2_{\alpha}$, the path 
$t \mapsto \Delta_{g_t, \alpha}u$ is analytic, and hence 
$t \mapsto \Delta_{g_t, \alpha}$ is a holomorphic family of type (A)
in the sense of Kato.
Rellich's Lemma implies that $\Delta_{g_t, \alpha}$ has compact 
resolvent.  Therefore, the claim follows from Theorem 3.9
in Chapter VII of \cite{Kato}.
\end{proof}

We will let $\lambda^{\alpha}_{j,t}$ denote the eigenvalue 
associated to the eigenfunction $\phi^{\alpha}_{j,t}$.
Since $t \mapsto \phi^{\alpha}_{j,t}$ is analytic,  the path 
$t \mapsto \lambda^{\alpha}_{j,t}$ is also analytic.
Analyticity implies that two eigenvalue branches $\lambda^{\alpha}_{j,t}$ and 
$\lambda^{\beta}_{k,t}$ are either equal to each other for
countably many $t$  or they coincide for all values of $t$.
To be precise, let $\Lambda >0$ and define
$I_{\alpha, \beta}(\Lambda) \subset (-\delta, \delta)$  to be the set of 
$t$ such that ${\rm \spec}(\Delta_{g_t, \alpha}) \cap [0, \Lambda] 
\cap {\rm \spec}(\Delta_{g_t, \beta}) \neq \emptyset$. 
Here we assume that $\alpha \neq \beta$ and hence 
$H^0_{\alpha} \perp H^0_{\beta}$ for each $t$.

\begin{lem}
\label{lem:countable-coincidence}
Either $I_{\alpha,\beta}(\Lambda)$ is countable 
or there exists $j,k$ such that 
$\lambda^{\alpha}_{j,t}= \lambda^{\beta}_{k,t}$ for each $t$.
\end{lem}

\begin{proof}
Suppose that there does not exist $j$ and $k$ so that 
$\lambda^{\alpha}_{j,t}= \lambda^{\beta}_{k,t}$ for all
$|t| < \delta$.  Real-analyticity implies that for each $j$ and $k$,
the set $\left\{ t \in (-\delta, \delta) \, :\, \lambda^{\alpha}_{j,t} =  \lambda^{\beta}_{k,t} \right\}$
is countable. Hence the union of these sets over $j,k$ is countable.
\end{proof}

\begin{proof}[Proof of Theorem \ref{thm:perturb} ]
Fix $\alpha \neq \pm \beta$.
For each $\Lambda >0$, let  $\Dcal(\Lambda)=\Dcal_{\alpha, \beta}(\Lambda)$
denote the set of $g \in \Mcal^T$
so that ${\rm \spec}(\Delta_{g, \alpha})\cap [0, \Lambda] \cap {\rm \spec}(\Delta_{g, \beta}) 
= \emptyset$. It suffices to show that $\Dcal(\Lambda)$ is open and dense in $\Mcal^T$
for each $\Lambda$.

To show that $\Dcal(\Lambda)$ is open will show that 
the complement is closed.
Suppose that there exist a sequence $g_m \in \Mcal^T$ 
and a sequence of $L^2$ unit norm eigenfunctions 
$u_m^{\alpha}$ of $\Delta_{g_m, \alpha}$ and 
$u_m^{\beta}$ of $\Delta_{g_m, \beta}$ so that 
$u_m^{\alpha}$ and $u_m^{\beta}$ have the same eigenvalue $\lambda_m \leq \Lambda$.
By passing to a subsequence if necessary, we may assume that 
$\lambda_m$ converges to $\lambda$ and that $u_m^{\alpha}$ (resp.  $u_m^{\beta}$)
converges to a unit norm eigenfunction $u^{\alpha}$ of $\Delta_{g, \alpha}$
(resp. $u^{\beta}$ of $\Delta_{g, \beta}$) with (the same) eigenvalue $\lambda$.
Hence $g$ lies in the complement of $\Dcal(\Lambda)$.
Thus, $\Dcal(\Lambda)$ is open.

To show that $\Dcal(\Lambda)$ is dense, we first show that $\Dcal(\Lambda)$ 
is nonempty.
Since $\alpha \neq \pm \beta$, there exists $j, k \in \{1, \ldots, d\}$
such that $\alpha_j \cdot \alpha_k \neq \beta_j\cdot \beta_k$.
We consider separately the cases where $j=k$ and $j \neq k$.

Suppose $j=k$. Let $g_0$ in $\Mcal^T$ be a metric on $M$ such that 
$g_0(\partial_i, \partial_j) =\delta_{ij}$ for each $i,j$.
Define the 1-form $\omega_j$ by $\omega_j(X)= g_0(X,\partial_j)$. 
If $u \in H^1_{\alpha}$, then by (\ref{eqn:vert-derivative})
\begin{equation}
\label{eqn:eval-vector}
\omega_j(\nabla_{g_0} u)~ 
= \partial_j u~ 
=~
\alpha_j \cdot 
\left(\begin{array}{cc}
0 & 1 \\
-1 & 0 
\end{array}
\right)
\cdot 
u.
\end{equation}
Define a path $t \mapsto g_t$ of $(2,0)$ tensors by
\begin{equation}
\label{eqn:i=j}
g_t~ 
=~
(1+t) \cdot g_0~
-~
t \cdot n \cdot \omega_j \otimes \omega_j.
\end{equation}
There exists $\delta>0$ so that $g_t$ is a Riemannian metric for $|t| < \delta$. 
A straightforward computation shows that if $d\nu_t$ is the measure
associated to $g_t$ then 
\[  
d \dot{\nu}~ 
=~ 
\frac{n}{2} \cdot \left(1\, -\, g_0(\partial_j, \partial_j) \right)\, d \nu_{0}.
\]

If $\Dcal(\Lambda)$ were empty, then we would have $I(\lambda)=(-\delta, \delta)$,
and hence Lemma \ref{lem:countable-coincidence} 
implies that there exist analytic paths of unit norm eigenfunctions
$t \mapsto u^\alpha_t$ and $t \mapsto u^\beta_t$ with common 
eigenvalue $\lambda_t$ so that for each $t \in (-\delta, \delta)$, the 
function $u^\alpha_t$  (resp. $u^{\beta}_t$) is an eigenfunction
of $\Delta_{\alpha,t}$ (resp. $\Delta_{\beta,t}$) 
with (the same) eigenvalue $\lambda_t$.
For the metric family defined in (\ref{eqn:i=j}) we have 
$\dot{g}= g_0-n \cdot \omega_j \otimes \omega_j$, and thus, 
using (\ref{eqn:eval-vector}) and the fact that $\int |u^{\alpha}_0|^2= 1$, 
we would find 
that formula (\ref{eqn:pert}) specializes to
\begin{equation*}
\dot{\lambda}~ 
=~
- \lambda_0~
+~
n \cdot \alpha_j^2.
\end{equation*}
The same formula would also hold with $\alpha$ replaced by $\beta$, and so 
we would find that $\alpha_j^2= \beta_j^2$, a contradiction.
Therefore $\Dcal(\Lambda)$ is not empty.

The proof in the case when $j \neq k$ is similar.
In this case consider the path of metrics
\begin{equation}
g_t~ 
=~
g_0~
-~
t 
\cdot 
\left( \omega_j \otimes \omega_k + \omega_k \otimes \omega_j \right).
\end{equation}
We have $d \dot{\nu}= 2 g_0(\partial_j, \partial_k) \, d \nu_{0}$,
and so (\ref{eqn:pert}) specializes to
\[
\dot{\lambda}~
=~
\alpha_i \cdot \alpha_j
\]
as well as the same equation with $\alpha$
replaced by $\beta$. Hence 
$\alpha_j \cdot \alpha_k = \beta_j \cdot \beta_k$,
a contradiction. Thus $\Dcal(\Lambda)$ is nonempty in this case as well.

To see that $\Dcal(\Lambda)$ is dense, 
given some other $g \in \Mcal^T$, consider linear path 
$g_t:= t \cdot m + (1-t) \cdot g_0$, $0 \leq t \leq 1$, in $\Mcal^T$
that joins $g_0$ to $g$. Because $g_0 \in \Dcal(\Lambda)$,
Lemma \ref{lem:countable-coincidence} implies that 
the set of $t$ such that $g_0 \notin \Dcal(\Lambda)$
is countable. Therefore, in every neighborhood of $g$ there exists
a metric $g'$ that lies in $\Dcal(\Lambda)$.
\end{proof}

\section{Generic simplicity for $\Delta_{g,\alpha}$.} 
\label{sec:simplicity}

In this section, we finish the proof of part (1) of Theorem \ref{thm:main}. 
In particular, we adapt the method of K. Uhlenbeck \cite{Uhlenbeck} 
to prove that there exists a residual subset of metrics $g$ so that 
the dimension of $\ker(\Delta_{g,\alpha}-\lambda I)$ is at most two.

Fix $\alpha \in \mathbb{Z}^{d} \setminus \{0\}$. 
Since $\alpha \neq 0$, each eigenspace of $\Delta_{g, \alpha}$ has dimension at least two. Indeed, if $u \in \ker(\Delta_{\alpha,g}-\lambda I)$, then

\[
u^*~
:=~ 
\left(
\begin{array}{cc}
0 & 1  \\
-1 & 0
\end{array} 
\right) 
\cdot 
u
\]
also lies in $\ker(\Delta_{\alpha,g}-\lambda I)$ and $\int u \cdot u^* =0$.

Given a $T$-invariant metric $g$ on $M$, define 
$F_g: H_{\alpha}^k \times \Rbb^2 \to H_{\alpha}^{k-2}$ 
by 
\begin{equation}
\label{eqn:f-defn}
F_g(u, a, b)~ 
=~
\Delta_{g, \alpha}u~ 
-~
a \cdot u~
-~
b \cdot u^*.
\end{equation}

A calculation shows 
that given any $u=(u_1, u_2)^{t}\in H^{k}_\alpha$, 
we have $F_g(u, a, b)=0$ 
if and only if $u_1+iu_2$ is a complex eigenfunction of $\Delta_{g}$ with eigenvalue $a-bi$.
But the eigenvalues of $\Delta_g$ are real and so
the zero locus of $F_g$ has the following form:
\begin{equation}
\label{eqn:F_zero_locus}
F_{g}^{-1}(0)~ 
=~
\{(u, \lambda, 0)\,|\, (\Delta_{g, \alpha}-\lambda I) u=0\}.
\end{equation}

The following Lemma should be compared with 
Lemma 2.2 in \cite{Uhlenbeck}.

\begin{lem}
\label{lem:simplicity}
Let $u \neq 0$ be in 
$\ker(\Delta_{g, \alpha}-\lambda I)$. 
Then the dimension of $\ker(\Delta_{g, \alpha}- \lambda I)$ 
equals two if and only if  
$(dF_g)_{(u,\lambda, 0)}$ is surjective.
\end{lem} 

\begin{proof}
We have $F_g(u, \lambda, 0)=0$  if and only if
$u \in \ker(\Delta_{g,\alpha}-\lambda I)$.
From (\ref{eqn:f-defn}) we find that for each 
$(v,\mu, \nu) \in H^k_{\alpha} \times \Rbb^2 $   
\begin{equation}
\label{eqn:diff-fg}
(dF_g)_{(u, \lambda, 0)}(v,\mu, \nu)~
=~
(\Delta_{g, \alpha}- \lambda  I)\cdot v\, -\, \mu \cdot u-\nu\cdot u^*.
\end{equation}
Thus 
\begin{equation}
\label{eqn:direct-sum}
\im((dF_g)_{(u, \lambda,0)})~
=~
{\rm im}(\Delta_{g,\alpha} - \lambda I)~ +~ \Rbb \cdot u~ +~ \Rbb \cdot u^*. 
\end{equation}
Because $\Delta_{g, \alpha} -\lambda I$ is self-adjoint, 
the space $\ker(\Delta_{g, \alpha} -\lambda I)$ is the orthogonal 
complement of $\im(\Delta_{g, \alpha} -\lambda I)$.
Hence the sums on the right hand side of (\ref{eqn:direct-sum}) are direct,
and so the quotient 
space $\im((dF_g)_{(u, \lambda,0)})/ {\rm im}(\Delta_{g,\alpha} - \lambda I)$
is two dimensional. 

On the other hand,
$H^k_{\alpha}/ {\rm im}(\Delta_{g,\alpha} - \lambda I)$ is 
isomorphic to $\ker(\Delta_{g, \alpha}- \lambda I)$.
The claim follows. 
\end{proof}

Lemma \ref{lem:simplicity} reduces generic irreducibility 
to the assertion that $(dF_g)_{(u, \lambda,0)}$ is surjective
for the generic $g$ and each eigenpair
$(u, \lambda)$ of $\Delta_{g, \alpha}$.
The next step in Uhlenbeck's method is to apply  
`parametric transversality' to the family 
$(u,a,b,g) \mapsto F_g(u, a, b)$. 

Let $\Scal^T$ denote the real Hilbert space of real-valued 
Sobolev $H^s$ symmetric $(0,2)$-tensors on $M$ that are $T$-invariant.
The space $\Mcal^T$ of Riemannian metrics on $M$
is the open cone in $\Scal^T$ consisting of positive symmetric $(0,2)$ tensors.
Define 
$F: H^k_{\alpha} \times \Rbb^2  \times \Mcal^T  \to H^{k-2}_{\alpha}$
by 
\begin{equation}
\label{eqn:F-defn}
F(u,a, b ,g)~ 
:=~ 
F_{g} (u, a, b)~ 
=~
\Delta_{g, \alpha}u\, -\, a \cdot u\, -\, b \cdot u^*.
\end{equation}

The following is a variant of parametric transversality.
Let $\pi: H_{\alpha}^{k}
\times \Rbb^2 \times \Mcal^T \to \Mcal^T$
denote the natural projection, for the remainder of this section. 

\begin{prop}
\label{prop:parametric}
Let $\Ucal \subset \Mcal^T$ be an open set.  If for 
each $x\in F^{-1}(0) \cap \pi^{-1}(\Ucal)$,  
the differential $dF_{x}$ is surjective, 
then there exists a residual subset of $\Rcal \subset \Ucal$ so that 
for each $g \in \Rcal$ and each $(u, \lambda, 0) \in F_g^{-1}(0)$,
the differential $(dF_g)_{(u, \lambda, 0)}$ is surjective.
\end{prop}

\begin{proof}
Let $x\in F^{-1}(0) \cap \pi^{-1}(\Ucal)$.
By hypothesis, $dF_{x}$ is surjective, and hence the 
implicit function theorem\footnote{Note that since 
$H^{k}_{\alpha} \times \Rbb^2 \times \Scal^T$ 
is a Hilbert space, the kernel of $dF_x$ is complemented.} 
implies that there exists 
an open neighborhood $\Vcal_x \subset H^k_{\alpha}  \times \Rbb^2 \times \Mcal^T$ 
of $x$ such that $F^{-1}(0) \cap \Vcal_{x}$ is a
submanifold. Thus, there exists an open subset 
$\Wcal \subset H^k_{\alpha}  \times \Rbb^2  \times \Mcal^T$ 
containing $\Ucal$ so that $S:=F^{-1}(0) \cap \Wcal$ is a 
 submanifold of 
$H^k_{\alpha} \times \Rbb^2   \times \Mcal^T $. 
The tangent space to $S$ at
$x \in S$ is the kernel of $dF_x$.
The map $F_g$ is the restriction of $F$ to $\pi^{-1}(g)$ and
\begin{equation}
\label{eqn:diff-Fg}
(dF_g)_{(u, \lambda, 0)}(v,\mu, \nu)~
=~
(\Delta_{g, \alpha}- \lambda  I)\cdot v\, -\, \mu \cdot u-\nu \cdot u^*.
\end{equation}

Let $\pi^S:S \to \Mcal^T$ denote the restriction of $\pi$ to $S$.
We claim that if $x =(u, \lambda, 0, g)$ belongs to $S$, then 
\begin{enumerate}[label=(\alph*)]
\item  $d\pi^S_{x}$ is Fredholm, and

\item  
the differential $(dF_g)_{(u, \lambda,0)}$ is surjective 
   if and only if $(d\pi^S)_x$ is surjective.
\end{enumerate}
Indeed, let $x \in S$, and note that $(d\pi^S)_x$ is the restriction 
of $d\pi_x$ to $\ker(dF_x)$ and $(dF_g)_x$ is the restriction of $dF_x$
to $\ker(d\pi_x)$.
We have 
\begin{equation}
\label{eqn:ker=ker}
\ker(dF_x|_{\ker{d\pi_x}})~
=~
\ker(dF_x) \cap \ker(d \pi_x)~
=~
\ker(d\pi_x|_{\ker(dF_x)}).
\end{equation}
Because $\Delta_{g,\alpha}-\lambda I$ is self-adjoint, the 
operator $\Delta_{g, \alpha}-\lambda I$ is Fredholm with index equal to zero. 
It follows that $(dF_g)_x$ given by (\ref{eqn:diff-Fg}) is Fredholm.
Thus it follows from (\ref{eqn:ker=ker}) that $\ker((d\pi^S)_x)$ is finite.
Since $dF_x$  and $d\pi_x$ are both surjective, 
the following maps of quotient spaces are isomorphisms:
\begin{equation}
\label{eqn:quo-iso}
\frac{\Scal^T}{d \pi_x(\ker(dF_x))}
\longleftarrow \frac{ H_{\alpha}^k \times  \Rbb^2 \times
\Scal^T }{\left(\ker(dF_x)+ \ker(d\pi_x) \right)}
\longrightarrow \frac{H^{k-2}_{\alpha}}{dF_x(\ker(d\pi_x))}.
\end{equation}
Since $(dF_g)_x$ is Fredholm, the space 
on the right is finite dimensional. Thus, 
since $d\pi_x(\ker(dF_x))= (d\pi^S)_x(T_xS)$,
the latter has finite codimension in $\Scal^T$,
and therefore $(d\pi^S)_x$ is Fredholm and (a) is proven. 
Part (b) follows directly from the isomorphisms in (\ref{eqn:quo-iso}).

Part (a) implies that we may apply the Smale-Sard theorem 
to $\pi^S$ to obtain a residual set 
$\Rcal \subset \Ucal$ so that for each $g \in \Rcal$ and 
$x \in (\pi^S)^{-1}(g)$, the differential $(d\pi^S)_x$ is surjective.  
Part (b) then implies that $(dF_g)_x$ is surjective. 
\end{proof}

Let $u \in H^k_{\alpha}$ be an eigenfunction of $\Delta_{g, \alpha}$.  
Define $G_{u}: \Mcal^T \to H^{k-2}_{\alpha}$ 
by setting 
\begin{equation}
\label{eqn:phi}
G_u(g)~ 
=~
\Delta_{g,\alpha}u.
\end{equation}

\begin{prop} 
\label{prop:residual-formal}
Let $\Ucal$ be an open subset of $\Mcal^T$.
Suppose that for each $g \in \Ucal$ and each eigenfunction $u$ of  
$\Delta_{g, \alpha}$,  the subspace $(dG_u)(\Scal^T)$ 
is dense in the orthogonal complement of $u^*$. 
If $g \in \Ucal$ and $x \in F^{-1}(0) \cap \pi^{-1}(g)$,
then the differential $dF_x$ is surjective.
\end{prop}

\begin{proof}

Let $g \in \Ucal$ and $x \in F^{-1}(0) \cap \pi^{-1}(g)$. Then by \eqref{eqn:F_zero_locus}, there exists $u$ and $\lambda$ such that $x=(u, \lambda , 0, g)$ and $(\Delta_{g, \alpha}-\lambda I)u=0$. Moreover, from (\ref{eqn:F-defn}) we find that 
\begin{equation}
\label{eqn:dF}
dF_{x}(v, \mu, \nu, h)~
=~
(\Delta_{g,\alpha}-\lambda)v~
+~
(dG_u)_g(h)~
-~ 
\mu \cdot u -\nu \cdot u^*.
\end{equation}
In particular, $dF_{x}(0, 0,h) = (dG_u)_g(h)$, and so the
image of  $(dG_u)_g$ is contained in the image of $dF_{x}$. 
Therefore, the hypothesis implies that the image of $dF_{x}$ is dense
in $(u^*)^{\perp}$. 
Since $dF_{x}(0,0,-1, 0)= u^*$, we have 
$u^* \in \im(dF_x)$, and therefore the image of $dF_{x}$
is dense in $H^{k}_{\alpha}$.

Since $(dF_g)_{(u, \lambda)}$ is Fredholm, the 
space 
\begin{equation}
\label{eqn:Z}
Z~ 
:=~ 
dF_{x}  \left( H^k_{\alpha}\times \{0\} \times \{0\} \right)~
=~
(dF_g)_{(u, \lambda,0)} \left( H^k_{\alpha} \times \{0\} \right)
\end{equation}
has finite codimension in $H^k_{\alpha}$. In particular, there exists a
finite dimensnional subspace $Z^c \subset H^k_{\alpha}$ so that 
$Z \oplus Z^c = H^k_{\alpha}$.  
Because the image of $dF_x$ is dense in $H^k_{\alpha}$,
the projection of $\im(dF_x)$  onto $Z^c$ is dense in $Z^c$.
But $Z^c$ is finite dimensional, 
and so the projection of $\im(dF_x)$ equals $Z^c$. 
On the other hand, by (\ref{eqn:Z}), 
the image of $dF_x$ contains $Z$.
It follows that $\im(dF_x)\supset Z \oplus Z^c = H^k_{\alpha}$.
\end{proof}

So far, our exposition of Uhlenbeck's method has been somewhat general
in the sense that we have not yet used any specific properties of torus
invariant metrics. To verify that the image of $(dG_u)_x$ is dense in the orthogonal complement of $u^{*}$,
we will need to use properties of these metrics.

For each $g \in  \Mcal^T$ and $\alpha \in \Zbb^d$, 
define $b_{g,\alpha}: M \to \Rbb$ by
\begin{equation}
\label{eqn:b-m-alpha}
b_{g,\alpha}(x)~
:=~
\frac{n+2}{d}~\sum_{i,j=1}^d\,  g^{ij}(x) \cdot \alpha_i \cdot \alpha_j.
\end{equation}
Note that $b_{g,\alpha}$ is $T$-invariant.
Define $\Ucal_{\alpha}$ to 
be the set of metrics $g \in  \Mcal^T$ such that $b_{g,\alpha}$ is 
nonconstant on each open subset of $M$. The set $\Ucal_{\alpha}$ 
is open and dense in $\Mcal^T$. 

\begin{prop}
\label{prop:DPhi2-dense}
If $g \in \Ucal_{\alpha}$ and $u$ is an eigenfunction 
of $\Delta_{g, \alpha}$, then $(dG_u)_{g}(\Scal^T)$ is dense in 
the orthogonal complement of $u^{*}$.
\end{prop}

\begin{proof}
To prove the proposition, it suffices to show that $(dG_u)_{g}(\Scal^T)^{\perp}$ is contained in $\mathrm{span}(u^{*})$.

Let $u=(u_1, u_2)^t$ be an eigenfunction of $\Delta_{g, \alpha}$ and $v=(v_1, v_2)^t\in (dG_u)_{g}(\Scal^T)^{\perp}$. We use two pertubation formulae to show that $v$ is a constant multiple of $u^{*}$.

First consider the perturbation defined in Subsection \ref{Subsection:MetricTorusInvariant}. Let $\{\partial_i\}$ be the 
the natural frame for $\Vcal$ defined in (\ref{eqn:partial_j}). 
If $g_{jk}:= g(\partial_j, \partial_k)$ is the Gram matrix associated to 
this frame, then 
$\nabla^\Vcal_g \xi= \sum_{j,k} g^{jk} \cdot \partial_i \xi \cdot \partial_j$.
In particular, for each $u=(u_1, u_2)^{t}, v=(v_1, v_2)^{t} \in H^{k-2}_{\alpha}$, we find that 
\begin{equation}
\label{eqn:eval-weight1}
g\left( \nabla^\Vcal u_1, \nabla^\Vcal v_1 \right)~
=~
\sum_{jk} g^{jk} \cdot \partial_j u_1 \cdot \partial_k v_1~
=~
u_2 \cdot v_2 \cdot \sum_{jk} g^{jk} \cdot \alpha_j \cdot \alpha_k
\end{equation}
and 
\begin{equation}
\label{eqn:eval-weight2}
g\left( \nabla^\Vcal u_2, \nabla^\Vcal v_2 \right)~
=~
\sum_{jk} g^{jk} \cdot \partial_j u_2 \cdot \partial_k v_2~
=~
u_1 \cdot v_1 \cdot \sum_{jk} g^{jk} \cdot \alpha_j \cdot \alpha_k.
\end{equation}

Recall the construction of the analytic path $g_{a_t, b_t}$ 
in Subsection \ref{Subsection:MetricTorusInvariant}. For each choice 
of smooth $T$-invariant function $f:M \to \Rbb$, 
we set $b_t= 1+ t \cdot f$ and $a_t:= 1 - \frac{n-d-2}{d} \cdot t \cdot f$
and define $g_t := g_{a_t, b_t}$. Then $d \cdot \dot{a} + (n-d-2)f=0$
and so using (\ref{eqn:eval-weight1}) and (\ref{eqn:eval-weight2}) we find that 
Lemma \ref{lem:variation} implies that for each 
$v \in H^{k-2}_{\alpha}$, we have 
\[
\int_M \dot{\Delta} u \cdot v\, d \nu_g ~
=~
2 \int_M
f \cdot (u_1v_1+u_2v_2) 
\cdot 
\left( 
b_{g, \alpha}~
-~
\lambda \right)\, d \nu_g
\]
Note that $\dot{\Delta}u= (dG_u)_{g}(\dot{g})$ and so for each $(v_1, v_2)^{t}\in (dG_u)_{g}(\Scal^T)^{\perp}$ and each $T$-invariant function $f: M \to \Rbb$,
\begin{equation}
\label{eqn:Delta-dot-applied}
0=\int_M \left[ (dG_u)_{g}(\dot{g})\right]\cdot v\, d \nu_g ~
=~
2 \int_M
f \cdot (u_1v_1+u_2v_2) 
\cdot 
\left( 
b_{g, \alpha}~
-~
\lambda \right)\, d \nu_g.
\end{equation}
Since $(u_1v_1+u_2v_2) 
\cdot 
\left( 
b_{g, \alpha}~
-~
\lambda \right)$ is $T$-invariant, it must vanish on $M$ almost everywhere.
By the definition of $\Ucal_{\alpha}$, the zero locus
of the smooth function $b_{g,\alpha} - \lambda$ has measure zero. 
Therefore, $u_1v_1+u_2v_2$ vanishes on a set of full measure. 

 Take $k$ in $H_{\alpha}^{k-2}$ to be large enough such that 
 $v_1$ and $v_2$ are in $C^1$. Then $u_1v_1+u_2v_2=0$ on the entire manifold.
On points such that $u_2$ is nonzero, define
$h:=-v_1/u_2$. Then $h$ is a $C^1$, invariant function away 
from the zero locus of $u_2$ and $(v_1, v_2)=h\cdot(-u_2, u_1)=h\cdot u^{*}$. 

Now we use another perturbation formula to show that $h$ is constant.
Let $X$ be a vector field that is invariant under the torus action
and is orthogonal to each $T$ orbit.
Let $Y=\partial_j$ be the vector 
field defined in \eqref{eqn:partial_j}. The corresponding perturbation 
$g_{tX, tY}$ that is defined in Lemma \ref{lem:mixed-perturbation} 
is $T$-invariant for each $t$ close to zero.

By \eqref{eqn:mixed-perturbation}, for each 
$v=(v_1, v_2)^t\in (dG_u)_{g}(\Scal^T)^{\perp}$ we have
\begin{equation}
\label{eqn:mix-perturbation-h}
\begin{split}
0&=\int_M \left[ (dG_u)_{g}(\dot{g})\right]\cdot v\, d \nu_g=
\int_M \dot{\Delta} u_1 \cdot v_1+\dot{\Delta} u_2 \cdot v_2\, d\nu_0\\
&=
-\sum_{i=1}^2\int_M
\left(
(Xu_i) \cdot (Y v_i)~ +~ (Y u_i)  \cdot (X v_i)
\right)\, 
d\nu_0.
\end{split}
\end{equation}
Since $(v_1, v_2)=h\cdot(-u_2, u_1)$,  
\eqref{eqn:mix-perturbation-h} can be simplified using 
(\ref{eqn:vert-derivative}) to 
\[
0~
=~
\alpha_j \int_{M} (u_1^2+u_2^2) \cdot X(h) \, 
d\nu_0. 
\]
Since $\alpha_j \neq 0$ for some $j$, the integral equals zero. 
Moreover, observe that both $(u_1^2+u_2^2)$ and $X(h)$ are $T$-invariant
and $(u_1^2+u_2^2)$ vanishes on a set of measure zero. 
Because the vector field $X$ is an arbitrary $T$-invariant vector 
field that is orthogonal to the orbits of $T$, the $T$-invariant 
function $h$ is a constant function.

Therefore, $(dG_u)_{g}(\Scal^T)^{\perp}$ is 
contained in the span of $u^{*}$ and the proposition is proved.
\end{proof}

\begin{thm}
\label{thm:simple-on-weights}
Let $\alpha \neq 0$.
There exists a residual subset $\Ycal_{\alpha}$ of $\Mcal^T$ so that 
for each $g \in \Ycal_{\alpha}$,
each eigenspace of $\Delta_{g, \alpha}-\lambda I$ is
two dimensional.
\end{thm}

\begin{proof}
Combine Lemma \ref{lem:simplicity},
Proposition \ref{prop:parametric}, Proposition~\ref{prop:residual-formal},
and Proposition~\ref{prop:DPhi2-dense}. 
\end{proof}

Now we have enough to prove the first part of Theorem \ref{thm:main}.

\begin{proof}[Proof of part (1) of Theorem \ref{thm:main}]

By Theorem \ref{thm:perturb}, there exists a residual subset 
of $\Mcal^T$ such that for each $g$ in this set we have  
$\spec(\Delta_{\alpha}) \cap \spec(\Delta_{\beta})= \emptyset$
unless $\alpha = \pm \beta$. If $\alpha =\pm \beta$,
then $H^k_{\alpha}=H^k_{\beta}$ and so the eigenspaces coincide.

If $\alpha \neq 0$,
then by Theorem \ref{thm:simple-on-weights}, there exists a residual set $\Ycal_{\alpha}$ such that the eigenspaces of $\Delta_{g,\alpha}$ are two-dimensional.
If $\alpha=0$, then as in the proof of Proposition \ref{prop-Weyl},
one may identify $H^k_0$ with the Sobolev space of 
$H^k$ functions on the quotient $M/T$. The operator
$\Delta_{g,0}$ corresponds to a nonnegative second order elliptic 
operator on $M/T$, and it follows from Theorem 8 in \cite{Uhlenbeck} that 
for a residual set of $g$, this operator has simple spectrum.

Since an intersection of residual sets is residual, the claim follows.
\end{proof}


\section{Nodal sets}
\label{sec:nodal-sets}

In this section we prove part (2) of Theorem \ref{thm:main}
and we prove Theorem \ref{thm:nodal=2}.

\begin{prop}
\label{prop:nodal-complex}
Let $\alpha \in \Zbb^d \setminus \{0\}$.
There exists a residual subset of $\Mcal^{T}$ so that
if $g$ belongs to this set, then $0 \in \Rbb^2$ 
is a regular value of each $u \in \ker(\Delta_{g, \alpha}- \lambda I)$.
\end{prop}

The idea of the proof originates in \cite{Uhlenbeck}, and 
an analogue appears as Proposition 4.8 in \cite{Jung-Zelditch}.
The proof is based on an analysis of the set $Q$ consisting of 
$(u, \lambda, g) \in H^2_{\alpha} \times \Rbb \times \Mcal^T$
 such that the dimension of 
$\ker(\Delta_{g, \alpha}-\lambda I)$  equals two. 
By the discussion surrounding (\ref{eqn:F_zero_locus}),
if $(u,\lambda, \lambda', g)$ lies in the submanifold 
$F^{-1}(0) \subset H^2_{\alpha} \times \Rbb \times \Mcal^T$
then $\lambda'=0$, and hence $Q$ is naturally embedded in $F^{-1}(0)$.  
For each $g \in \Mcal^T$, 
each eigenspace of $\Delta_{g,\alpha}$
has dimension at least two, and so
the continuity of the spectrum implies that the 
embedding of $Q$ is open in $F^{-1}(0)$.

To prove Proposition \ref{prop:nodal-complex} we will
use the following lemma. 

\begin{lem}
\label{lem:E-dense}
Let $(u, \lambda, g) \in Q$,
and let $E$ be the subspace of $ \ker(\Delta_{g,\alpha} -\lambda I)^{\perp}$ 
consisting $v$ such that there exists
$(\mu,h) \in \Rbb \times \Tcal^T$ so that 
$(v, \mu, h)$ lies in the tangent space $T_{(u, \lambda, g)} Q$.
Then $E$ is dense in $ \ker(\Delta_{g,\alpha} -\lambda I)^{\perp}$.
\end{lem}

\begin{proof}
The tangent space $T_{(u, \lambda,g)} Q$ 
is the kernel of $dF_{(u, \lambda,0,g)}$ which by 
(\ref{eqn:dF}) is the space of 
$(v, \mu, 0,h) \in H^k_{\alpha} \times \Rbb \times \Mcal$
such that
\begin{equation}
\label{eqn:tangent-space}
(\Delta_{g, \alpha}-\lambda I) v~ 
+~
(dG_u)(h)~
-~
\mu \cdot u~
=~
0.
\end{equation}

For simplicity of notation, let $V:=\ker(\Delta_{g,\alpha} -\lambda I)^{\perp}$.
The operator $\Delta_{g, \alpha} - \lambda I$ is self-adjoint, and hence 
$V =\im(\Delta_{g, \alpha} - \lambda I)$.
Let $p: H^k_{\alpha} \to H^k_{\alpha}$
be the orthogonal projection onto $V$. 
Let $L := (\Delta_g - \lambda I)$ denote the restriction
of $\Delta_{g,\alpha}$ to $V$. The map $L$ is an isomorphism onto $V$
with bounded inverse $L^{-1}$.

We claim that the image of $-L^{-1} \circ p \circ (dG_u)_g$ is contained in $E$.
Indeed, suppose that 
\[
v~ 
=~
-  L^{-1} \circ p \circ (dG_u)_g(h).   
\]
Then 
\[
L v~ 
=~
- p \circ (dG_u)_g(h),  
\]
and so 
\[
(\Delta_{g, \alpha} - \lambda I) v~ 
=~
-(dG_u)_g(h)~ 
+
w
\]
for some $w \in \ker(\Delta_{g,\alpha}-\lambda I)= \Rbb \cdot u \oplus \Rbb \cdot u^*$.
Since the image of $(dG_u)_g$ lies in $(u^*)^{\perp}$ and 
$Lv \in V \subset (u^*)^{\perp}$,  we have $w= c \cdot u$ for some $c \in \Rbb$.
Thus,
\[
(\Delta_{g, \alpha} - \lambda I) v~ 
+~
(dG_u)_g(h)~ 
-
c \cdot u~
=~
0,
\]
and therefore $v \in E$ as claimed.

By Proposition \ref{prop:DPhi2-dense}, 
the image of $(dG_u)_g$ is dense in $(u^*)^{\perp}$,
and hence the image of $L^{-1} \circ p \circ  (dG_u)_g$ is dense in $V$.
Therefore $E$ is dense in $V$.
\end{proof}

\begin{proof}[Proof of Proposition \ref{prop:nodal-complex}]

Assume that $k$ is large enough so that $H_{\alpha}^k \subset C^1$
by the Sobolev embedding theorem.
Then the map $f: H^k_{\alpha} \times \Rbb \times \Mcal^T \times M \to  \Rbb^2$ by
\[
f(u, \lambda, g, x) = u(x)
\]
is $C^1$. The differential is given by
\begin{equation}
\label{differential-f}
df_{(u, \lambda, g, x)}(v, \mu,h,z)~
=~
du_x(z)~ +~ v(x).
\end{equation}
We now restrict $f$ to the submanifold $Q \times M$, and, abusing 
notation, we continue to denote the restriction by $f$.  
The differential of $f$ is still given by (\ref{differential-f})
where now $(v, \mu, h,z)$ belongs to the tangent space 
$T_{(u, \lambda,g,x)}(Q\times M)$. 

We claim that $df_{(u, \lambda,g,x)}$ is a surjection if $f(u, \lambda, g,  x) =0$.
If not, then there exists $(u, \lambda, g,  x)\in f^{-1}(0)$ and 
$a \in \Rbb^2$ so that for each $(v,\mu,h,z) \in T_{(u,\lambda,g,x)}(Q \times  M)$
we have 
\[ 
a \cdot df_{(u,\lambda,g,x)}(v,\mu,h,z)=0.
\]
In particular,  $a \cdot df_{(u,\lambda,g,x)}(v,\mu, h,0)= a \cdot v(x)=0$. 
for each $v \in E$ where the set $E$---defined in Lemma \ref{lem:E-dense}---
is dense in $V = \ker(\Delta_{g,\alpha}-\lambda I)^{\perp}$. 
We have chosen $k$ so that the
$H^k$ topology is stronger than the $C^0$ topology, and so $a \cdot v(x)=0$
for each $v \in E$ implies that $a \cdot v(x)=0$ for each $v \in V$.
Separately, since $f(u, \lambda,g,x)=0$, we have $u(x)=0$ and hence $u^*(x)=0$.
Thus $a \cdot w(x)=0$ for each $ w \in \ker(\Delta_{g,\alpha} -\lambda I)$, and 
in sum we find that
$a \cdot w(x)=0$ for each $w \in H^k_{\alpha} = V \oplus \ker(\Delta_{g,\alpha} -\lambda I)$. 

Let $w \in H^k_{\alpha}$ be such that $w(x)\neq 0$. If $a \neq 0$, then 
there exists a rotation $R \in SO(2)$ and a positive real number $c$ 
so that $a = c \cdot R w(x)$.  
But $c \cdot R w  \in H^k_{\alpha}$ and so using the above 
we have $0=a \cdot (c \cdot R w(x))= c \cdot a \cdot a$,
a contradiction. Hence $a=0$ and $df_{(u, \lambda, g, x)}$ is surjective
for each $(u, \lambda, g,  x)\in f^{-1}(0)$.

Let $\pi: Q \times M \to Q$ denote the natural projection.
Given $q=(u, \lambda, g)\in Q$ let $f_q$ denote the restriction
of $f$ to $\pi^{-1}(q)$. Because, $df_{(u, \lambda, g,x)}$ is surjective
for each $(u, \lambda, g,x) \in f^{-1}(0)$, the parametric transversality 
theorem applies to provide a residual subset $\Pcal \subset Q$, 
so that if $q=(u,\lambda,g)$ lies in $\Pcal$ then $d(f_q)_x=du_x$ 
is surjective for each $x \in u^{-1}(0)$.

Let $\Rcal$ be the set of $g \in \Mcal^T$ such that 
$\ker(\Delta_{g, \alpha} - \lambda I)$ is two dimensional
for each $\lambda \in \spec(\Delta_{g, \alpha})$.
The first part of Theorem \ref{thm:main}
implies that $\Rcal$ is residual in $\Mcal^T$. 
Since $\pi:Q \to \Mcal^T$ is open, the set $\pi(\Pcal)$ is residual in $\Mcal^T$,
and hence  $\pi(\Pcal) \cap \Rcal$ is residual in $\Mcal^T$.
\end{proof}

\begin{proof}[Proof of part (2) of \ref{thm:main}]
If an eigenfunction is $T$-invariant, then it descends
to an eigenfunction of an elliptic operator on $M/T$, and 
it follows from Theorem 8 in \cite{Uhlenbeck} that, for
a residual set of $g$, the nodal set of each $T$-invariant 
eigenfunction of $\Delta_g$ is smooth.

If an eigenfunction is not $T$-invariant and $g$ lies in the
residual subset obtained in the proof of part (1) of Theorem 
\ref{thm:main}, then the eigenfunction is the first coordinate $u_1$
of a vector-valued function $u=(u_1,u_2)^t \in H^k_{\alpha}$. If 
$(u_1(x), u_2(x))^t=0$, then Proposition \ref{prop:nodal-complex}
implies that $d(u_1)_x$ is surjective.  If $u_1(x)=0$ but 
$u_2(x) \neq 0$, then from (\ref{eqn:partial_j}) we have  
$(\partial_j u)(x)= \alpha_j \cdot u_2$, and since $\alpha \neq 0$,
the differential $d(u_1)_x$ is surjective. Therefore, 
the implicit function theorem implies that $u_1^{-1}(0)$ 
is smooth.
\end{proof}

The following generalizes parts (2) and (3) of Theorem 1.5 in \cite{Jung-Zelditch}.

\begin{thm}
\label{thm:connected-nodal-sets}
Suppose that $\dim(M/T) \geq 2$ and let $\alpha \in \Zbb^d  \setminus \{0\}$.
There exists a residual subset $\mathcal{Y}_\alpha$ in $\Mcal^{T}$ so that
if $g$ lies in $\Ycal_{\alpha}$ 
and $u=(u_1,u_2)^t$ is an eigenfunction of $\Delta_{g, \alpha}$
whose values include zero, 
then the nodal set of $u_1$ (resp. $u_2$) is a connected smooth manifold,
and its complement consists of exactly two connected 
components.
\end{thm}

\begin{proof}
Let $u=(u_1,u_2)^t$ be an eigenfunction of $\Delta_{g, \alpha}$ 
and let $\pi: M \to M/T$ be the standard submersion.
By Proposition \ref{prop:nodal-complex} and the implicit 
function theorem, we find that the set $u^{-1}(0)$ 
is a smooth submanifold of $M$ with dimension $n-2$.
Moreover, since $u(\theta^{-1}\cdot x) = R_{\alpha \cdot \theta} \cdot u(x)$
for each $\theta$ in the $d$-dimensional torus $T$, 
the action of $T$ on $M$ preserves $u^{-1}(0)$.
In particular, $\pi(u^{-1}(0))$ is of real codimension two, and hence 
$\pi(M)\setminus \pi(u^{-1}(0))$ is connected.
Moreover, if $x \in \pi(u^{-1}(0))$, then $\pi^{-1}(x)$ is 
a connected torus embedded in $u^{-1}(0)$.

Define $N:=u_1^{-1}(0) \setminus u^{-1}(0)$. We claim that 
the restriction of $\pi$ to $N$ is a submersion onto its image 
$\pi(N)= \pi(M) \setminus \pi(u^{-1}(0))$. Indeed, since $\alpha \neq 0$, 
there exists $i$ such that $\alpha_i \neq 0$. 
Recall from \S  \ref{sec:prelim} that $\partial_i u = \alpha_i \cdot u^*$
and in particular $\partial_i u_1 =  \alpha_i \cdot u_2$.
Thus, we have $\partial_{i} u_1(x) \neq 0$
and $u_1(x)=0$ if and only if $u(x) \neq 0$.
In particular, the vector field $\partial_i$ is transverse to $N$
and since $\pi$ is a submersion, the restriction $\pi|_N$ is also
a submersion. 
Each fiber of $\pi|_N$ is  diffeomorphic to the disjoint union of $2m$ copies
of the $d-1$ dimensional torus where $m = \gcd(\alpha_1,\ldots, \alpha_d)$.

By hypothesis, $u^{-1}(0)$ is nonempty and hence contains some $y_0$. 
Because $\alpha \neq 0$, the eigenfunction
$u$ vanishes on the fiber $\pi^{-1}(x_0)$ where $x_0 = \pi(y_0)$. 
The fiber $\pi^{-1}(x_0)$ is a torus and hence is connected.
Thus, to show that $u_1^{-1}(0)$ is connected, it suffices 
to show that, for each $y \in N$, the connected component 
$K_y$ of $u_1^{-1}(0)$ 
that contains $y$ also contains a point in $u^{-1}(x_0)$.

Since $\pi(M)$ is connected and $\pi(u^{-1}(0))$ is a smooth 
closed submanifold of codimension $2$, there exists a path 
$\gamma:[0,1] \to \pi(M)$ such that $\gamma(0)= \pi(y)$, $\gamma(1)=x_0$ 
and $\gamma(t) \notin \pi(u^{-1}(0))$ for each $t <1$. 
Choose a horizontal distribution for $\pi: N \to \pi(N)$.
Let $\tgamma:[0, 1) \to N$ denote the horizontal lift of $\gamma|_{[0,1)}$.

Suppose that $t_k \in [0,1)$  converges to $1$.
Since $M$ is compact, the sequence $\tgamma(t_k)$ has a 
convergent subsequence that converges to some $z \in M$. 
On the other hand, $\pi(\tgamma(t_k))= \gamma(t_k)$ converges to $x_0$,
and so $z \in \pi^{-1}(x_0)$. Hence $K_y$ intersects $\pi^{-1}(x_0)$. 
Therefore, $u_1^{-1}(0)$ is connected. 

Finally, since $\alpha \neq 0$, the eigenfunction $u_1$
takes on both positive and negative values, and hence
the set $M \setminus u^{-1}_{1}(0)$ has at least two components.
In particular, $u^{-1}_{1}(0)$ separates $M$ and a classical 
result in differential topology gives that $M \setminus u^{-1}_{1}(0)$
has exactly two connected components.\footnote{See, for example, 
Lemma 4.4.4 in \cite{Hirsch}.}
\end{proof}

Let $\Ycal = \bigcap_{\alpha \neq 0} \Ycal_{\alpha}$.

\begin{coro}
\label{coro:vanish-orbit}
Suppose $\dim(M/T) \geq 2$.
Let $g$ belong to the residual subset $\Ycal \subset \Mcal ^T$.
If $\phi: M \to \Rbb$ is a non-invariant eigenfunction of $\Delta_g$
such that $\phi$ vanishes on some $T$-orbit, then the nodal set 
$\phi^{-1}(0)$ is a connected smooth hypersurface 
and its complement consists of two components.
\end{coro}

\begin{proof}
Since $g \in \Ycal$ and $\phi$ is non-invariant, there exists 
$\alpha \in \Zbb \setminus \{0\}$ and  
$u =(u_1,u_2)^t \in H^k_{\alpha}$ such that $\phi=u_1$.
If $\phi(x)=u_1(x)=0$ for each $x$ in some orbit of $T$, 
then $\partial_j u_1(x)=0$ vanishes for each $x$ in this orbit.
Thus, if $\alpha_j \neq 0$, then from (\ref{eqn:vert-derivative})
we find that $u_2(x)= \alpha_j^{-1} \cdot \partial_j u_1(x)$
for $x$ in the orbit. The claim now follows from 
Theorem \ref{thm:connected-nodal-sets}.
\end{proof}


\section{Circle bundles}
\label{sec:oriented-circle}

In this section, we explain how the space $H^k_{\alpha}$, $\alpha \neq 0$,
may be identified with sections of an oriented 
rank two real vector bundle, and we translate the results of \S 
\ref{sec:nodal-sets} into this language. Then we briefly discuss
principal $SO(2)$ bundles and the natural group structure 
on the set such bundles over a fixed base.    
Finally, we prove Theorem \ref{thm:nodal=2}.

Let $\alpha \neq 0$, and define the vector bundle $\pi: E^{\alpha} \to M/T$ 
as the equivariant quotient of the trivial bundle $M \times \Rbb^2 \to M$ 
where the action of $T$ on $M \times \Rbb^2 \to M$ 
is given by 
\begin{equation}
\label{eqn:action}
\theta \cdot (x,w)~
:=~
(\theta \cdot x, R_{\alpha \cdot \theta} \cdot w)
\end{equation}
with $R_{\psi}$ as in (\ref{eqn:R}).
A section $x \to (x,u(x))$ of the trivial bundle 
$M \times \Rbb^2 \to M$ descends to a section $\sigma_u$ of $E^{\alpha}$
if and only if 
\begin{equation}
\label{eqn:section}
(\theta \cdot x, u(\theta \cdot x))~
= 
(\theta \cdot x, R_{\alpha \cdot \theta} \cdot u(x))
\end{equation}
for each $x \in M$ and $\theta \in T$. 
In other words, the map $u \mapsto \sigma_u$ is an isomorphism
from $H^k_{-\alpha}$ onto the space 
of $H^k$  sections of $E^{\alpha}$.

Note that the action of $SO(2)$ provides an orientation of each 
fiber of $E^{\alpha}$. 
By removing the zero section of $E^{\alpha}$ and quotienting each fiber
by the action of $\Rbb^+$, we obtain an oriented circle bundle
$M^{\alpha} \to M/T$.

\begin{prop}
\label{prop:oriented}
If the oriented circle bundle $M^{\alpha}$ is nontrivial, 
then each 
eigenfunction $u$ of $\Delta_{g, \alpha}$ vanishes on some torus orbit.
\end{prop}

\begin{proof}
Let $u \in H^k_{\alpha}$ be an eigenfunction of $\Delta_{g,\alpha}$.
By elliptic regularity, the function $u$ is smooth, and it corresponds
to a smooth section $\sigma_u$ of $E^{\alpha}$. If $u(x) \neq 0$
for all $x \in M$, then $\sigma_u(b) \neq 0$ for each $b \in M/T$. 
Thus $\sigma_u(b)$ would define a global section of $M^{\alpha}$. 
But then $M^{\alpha}$ would be trivial, contradicting our assumption. 
\end{proof}

\begin{prop}
\label{prop:torus-nontrivial}
If each oriented circle bundle $M^{\alpha} \to M/T$
associated to the torus bundle 
$M \to M/T$ is nontrivial, then there exists a residual subset 
of $\Mcal^T$ such that for each $g$ in this set and each 
non-invariant eigenfunction $\phi$ of $\Delta_g$, the 
nodal set $\phi^{-1}(0)$ is a connected smooth submanifold
whose complement has exactly two components. 
\end{prop}

\begin{proof}
Combine Corollary \ref{coro:vanish-orbit} and Proposition \ref{prop:oriented}.
\end{proof}

\begin{proof}[Proof of Theorem~\ref{thm:nodal=2}]
The Euler class $e(N) \in H_2(B,\Zbb)$ is a complete 
invariant for oriented circle bundles $N$ over $B$ \cite{Morita}. 
In particular, if $M$ is a nontrivial oriented circle bundle,
then $e(M) \neq 0$. Thus by hypothesis, the class $e(M)$
has infinite order in $H_2(B,\Zbb)$. 
If $\alpha \in \Zbb$, then $e(M^{\alpha})=  \alpha \cdot e(M)$.
Thus, if $\alpha \neq 0$, then $e(M^{\alpha}) \neq 0$
and so $M^{\alpha}$ is not trivial. Thus, the claim 
follows from Proposition \ref{prop:torus-nontrivial}.
\end{proof}


\end{document}